\documentclass[12pt]{article}
\usepackage{amsmath,latexsym,amssymb}
\usepackage{enumerate}
\usepackage{amsthm}
\usepackage{epsfig,graphicx}
\usepackage{epic}

\newcommand\A{\mathcal{A}}
\newcommand\NN{\mathbb{N}}
\newcommand\RR{\mathbb{R}}
\newcommand\ZZ{\mathbb{Z}}

\newcommand\DD{\mathcal{D}}
\newcommand\FF{\mathcal{F}}

\newcommand\UU{\mathcal{U}}

\newcommand\eps{{\varepsilon}}

\renewcommand{\mod}{\operatorname{mod}}
\DeclareMathOperator\graph{Graph}

\DeclareMathOperator\Dim{Dim}

\newtheorem{theorem}{Theorem}[section]

\newtheorem{corollary}[theorem]{Corollary}

\newtheorem{lemma}[theorem]{Lemma}

\newtheorem{proposition}[theorem]{Proposition}

\newcommand{\address}{Address: Department of Mathematics, University of North Texas, 1155 Union Circle \#311430, Denton, TX 76203-5017, USA; E-mail: allaart@unt.edu}

\title{Differentiability of a two-parameter family of self-affine functions}
\author{Pieter C. Allaart \footnote{\address}}

\begin{document}

\maketitle

\begin{abstract}

This paper highlights an unexpected connection between expansions of real numbers to noninteger bases (so-called {\em $\beta$-expansions}) and the infinite derivatives of a class of self-affine functions. Precisely, we extend Okamoto's function (itself a generalization of the well-known functions of Perkins and Katsuura) to a two-parameter family $\{F_{N,a}: N\in\NN, a\in(0,1)\}$. We first show that for each $x$, $F_{N,a}'(x)$ is either $0$, $\pm\infty$, or undefined. We then extend Okamoto's theorem by proving that for each $N$, depending on the value of $a$ relative to a pair of thresholds, the set $\{x: F_{N,a}'(x)=0\}$ is either empty, uncountable but Lebesgue null, or of full Lebesgue measure. We compute its Hausdorff dimension in the second case. 


The second result is a characterization of the set $\DD_\infty(a):=\{x:F_{N,a}'(x)=\pm\infty\}$, which enables us to closely relate this set to the set of points which have a unique expansion in the (typically noninteger) base $\beta=1/a$. Recent advances in the theory of $\beta$-expansions are then used to determine the cardinality and Hausdorff dimension of $\DD_\infty(a)$, which depends qualitatively on the value of $a$ relative to a second pair of thresholds.




\end{abstract}

\renewcommand{\bottomfraction}{0.4}

\section{Introduction}

The aim of this paper is to investigate the differentiability of a two-parameter family of self-affine functions, constructed as follows.
Fix a positive integer $N$ and a real parameter $a$ satisfying $1/(N+1)<a<1$, and let $b$ be the number such that $(N+1)a-Nb=1$. Note that $0<b<a$. Let $x_i:=i/(2N+1)$, $i=0,1,\dots,2N+1$, and for $j=0,1,\dots,N$, put
$y_{2j}:=j(a-b)$, and $y_{2j+1}:=(j+1)a-jb$.
Now set $f_0(x):=x$, and for $n=1,2,\dots$, define $f_n$ recursively on each interval $[x_i,x_{i+1}]$ ($i=0,1,\dots,2N$) by
\begin{equation}
f_n(x):=y_i+(y_{i+1}-y_i)f_{n-1}\big((2N+1)(x-x_i)\big), \qquad x_i\leq x\leq x_{i+1}.
\end{equation}
Each $f_n$ is a continuous, piecewise linear function from the interval $[0,1]$ onto itself, and it is easy to see that the sequence $(f_n)$ converges uniformly to a limit function which we denote by $F_{N,a}$. This function $F_{N,a}$ is again continuous and maps $[0,1]$ onto itself. It may be viewed as the self-affine function ``generated" by the piecewise linear function with interpolation points $(x_i,y_i)$, $i=0,1,\dots,2N+1$. When $N=1$ we have Okamoto's family of self-affine functions \cite{Okamoto}, which includes Perkins' function \cite{Perkins} for $a=5/6$ and the Katsuura function \cite{Katsuura} for $a=2/3$; see also Bourbaki \cite{Bourbaki}. Figure \ref{fig:construction} illustrates the above construction for $N=1$; graphs of $F_{1,a}$ for two values of $a$ are shown in Figure \ref{fig:Okamoto-graphs}; and Figure \ref{fig:5-part-function} illustrates the case $N=2$.

\begin{figure}[b] 
\begin{center}
\begin{picture}(360,150)(0,15)
\put(30,20){\line(1,0){126}}
\put(30,20){\line(0,1){126}}
\put(30,146){\line(1,0){126}}
\put(156,20){\line(0,1){126}}
\put(27,15){\makebox(0,0)[tl]{$0$}}
\put(72,18){\line(0,1){4}}
\put(62,15){\makebox(0,0)[tl]{$1/3$}}
\put(114,18){\line(0,1){4}}
\put(104,15){\makebox(0,0)[tl]{$2/3$}}
\put(154,15){\makebox(0,0)[tl]{$1$}}
\put(17,25){\makebox(0,0)[tl]{$0$}}
\put(28,104){\line(1,0){4}}
\put(16,107){\makebox(0,0)[tl]{$a$}}
\put(28,62){\line(1,0){4}}
\put(-2,67){\makebox(0,0)[tl]{$1-a$}}
\put(19,150){\makebox(0,0)[tl]{$1$}}
\put(30,20){\line(1,2){42}}
\put(72,104){\line(1,-1){42}}
\put(114,62){\line(1,2){42}}
\put(60,122){\makebox(0,0)[tl]{$f_1$}}
\thicklines
\dottedline{4}(30,20)(156,146)
\thinlines
\put(230,20){\line(1,0){126}}
\put(230,20){\line(0,1){126}}
\put(230,146){\line(1,0){126}}
\put(356,20){\line(0,1){126}}
\put(227,15){\makebox(0,0)[tl]{$0$}}
\put(272,18){\line(0,1){4}}
\put(262,15){\makebox(0,0)[tl]{$1/3$}}
\put(314,18){\line(0,1){4}}
\put(304,15){\makebox(0,0)[tl]{$2/3$}}
\put(354,15){\makebox(0,0)[tl]{$1$}}
\put(217,25){\makebox(0,0)[tl]{$0$}}
\put(228,104){\line(1,0){4}}
\put(216,107){\makebox(0,0)[tl]{$a$}}
\put(228,62){\line(1,0){4}}
\put(198,67){\makebox(0,0)[tl]{$1-a$}}
\put(219,150){\makebox(0,0)[tl]{$1$}}
\put(230,20){\line(1,4){14}}
\put(244,76){\line(1,-2){14}}
\put(258,48){\line(1,4){14}}
\put(272,104){\line(1,-2){14}}
\put(286,76){\line(1,1){14}}
\put(300,90){\line(1,-2){14}}
\put(314,62){\line(1,4){14}}
\put(328,118){\line(1,-2){14}}
\put(342,90){\line(1,4){14}}
\put(260,122){\makebox(0,0)[tl]{$f_2$}}
\thicklines
\dottedline{4}(230,20)(272,104)
\dottedline{4}(272,104)(314,62)
\dottedline{4}(314,62)(356,146)
\end{picture}
\end{center}
\caption{The first two steps in the construction of $F_{1,a}$}
\label{fig:construction}
\end{figure}
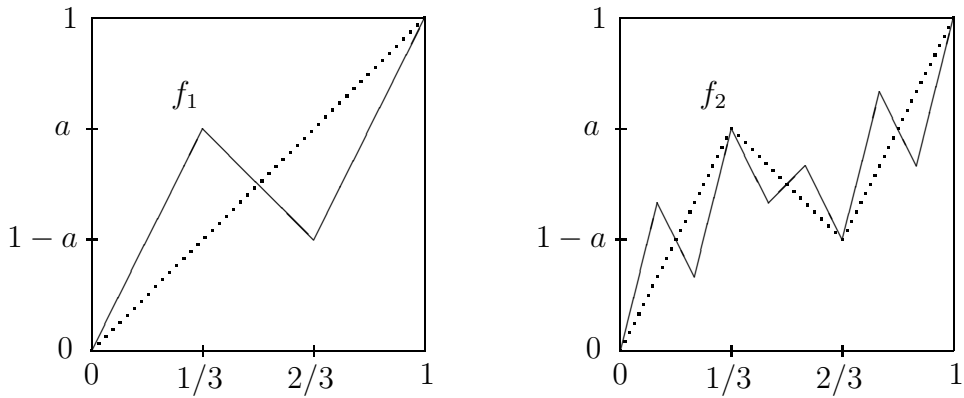

\begin{figure}
\begin{center}
\epsfig{file=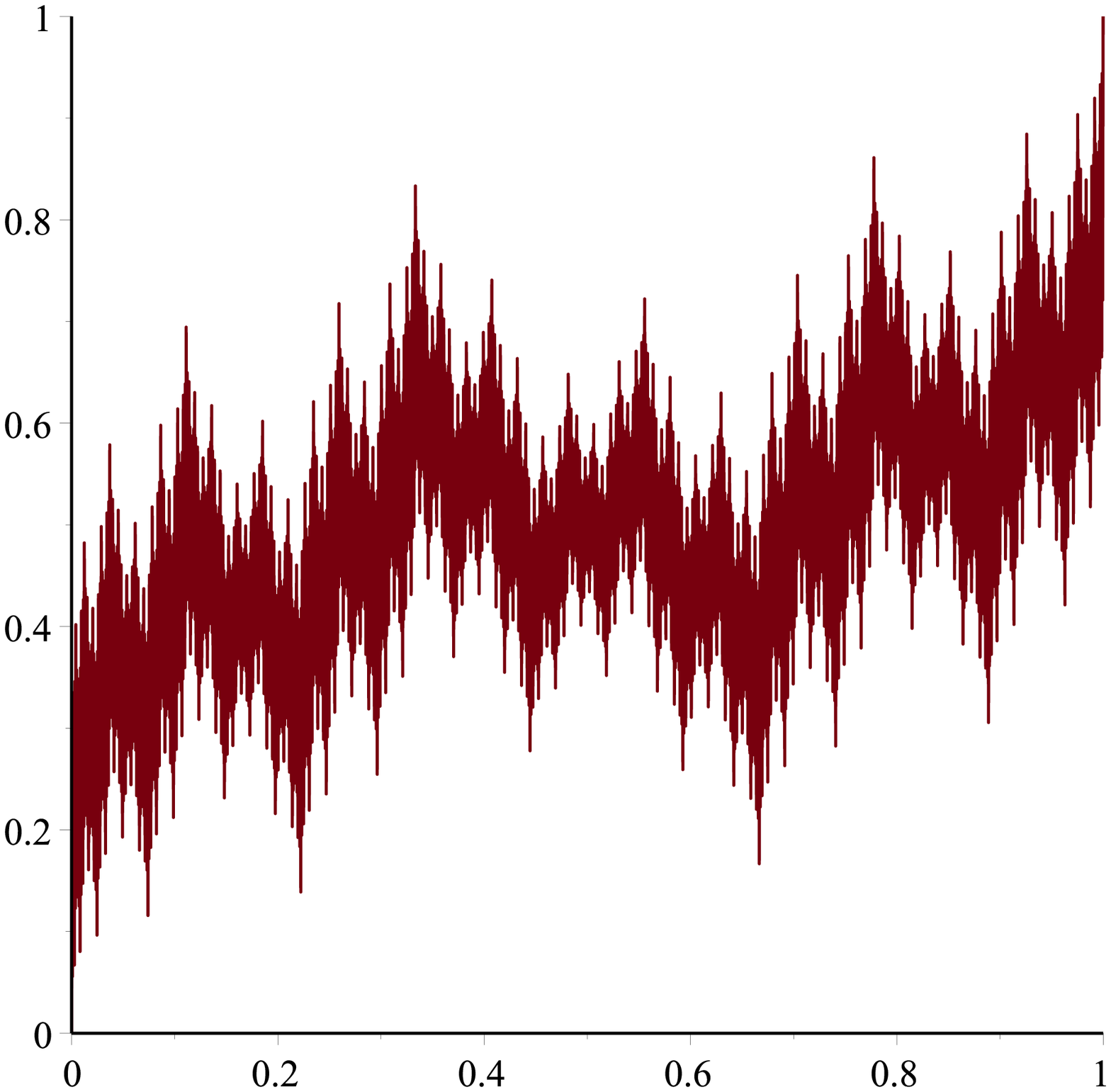, height=.25\textheight, width=.35\textwidth} \qquad\quad
\epsfig{file=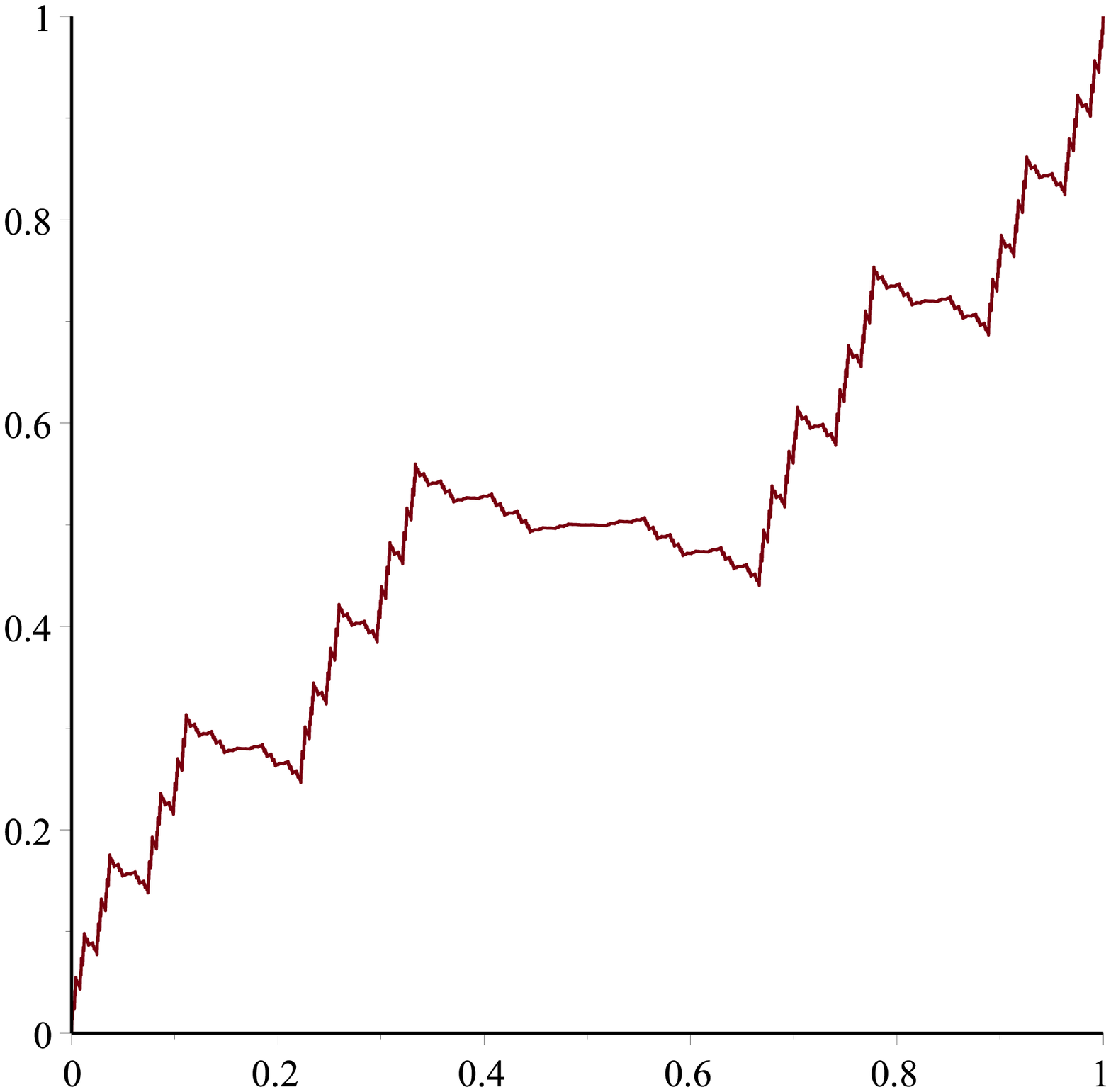, height=.25\textheight, width=.35\textwidth}
\caption{Graph of $F_{1,a}$ for $a=5/6$ (Perkins' function; left) and $a=\hat{a}_\infty\doteq .5598$ (right).}
\label{fig:Okamoto-graphs}
\end{center}
\end{figure}

\begin{figure}
\begin{center}
\begin{picture}(160,150)(0,0)
\put(30,20){\line(1,0){125}}
\put(30,20){\line(0,1){125}}
\put(30,145){\line(1,0){125}}
\put(155,20){\line(0,1){125}}
\put(30,20){\line(1,3){25}}
\put(55,95){\line(1,-2){25}}
\put(80,45){\line(1,3){25}}
\put(105,120){\line(1,-2){25}}
\put(130,70){\line(1,3){25}}
\put(55,18){\line(0,1){4}}
\put(80,18){\line(0,1){4}}
\put(105,18){\line(0,1){4}}
\put(130,18){\line(0,1){4}}
\put(28,95){\line(1,0){4}}
\put(28,45){\line(1,0){4}}
\put(27,15){\makebox(0,0)[tl]{$0$}}
\put(153,15){\makebox(0,0)[tl]{$1$}}
\put(17,25){\makebox(0,0)[tl]{$0$}}
\put(19,149){\makebox(0,0)[tl]{$1$}}
\put(16,97){\makebox(0,0)[tl]{$a$}}
\put(0,50){\makebox(0,0)[tl]{$a-b$}}
\end{picture}
\qquad\qquad
\epsfig{file=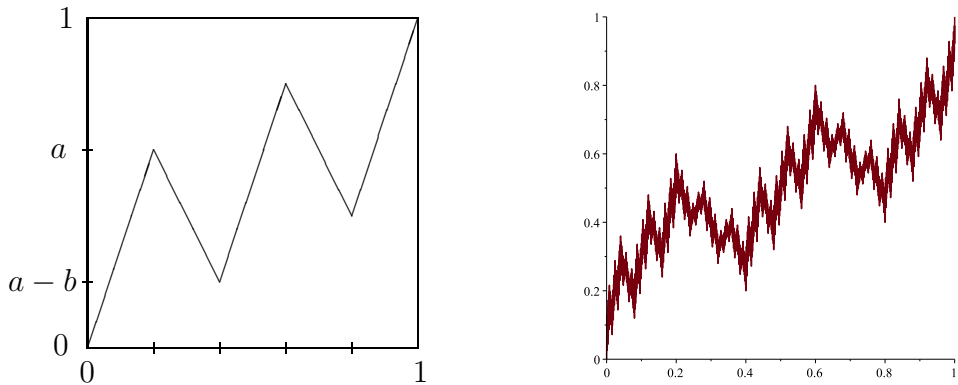, height=.26\textheight, width=.35\textwidth} \qquad\quad
\caption{The generating pattern and graph of $F_{2,a}$, shown here for $a=0.6$.}
\label{fig:5-part-function}
\end{center}
\end{figure}


The restriction $a>1/(N+1)$ is not necessary; when $a=1/(N+1)$ we have $b=0$ and $F_{N,a}$ is a generalized Cantor function. When $0<a<1/(N+1)$, we have $b<0$ and $F_{N,a}$ is strictly singular. Since the differentiability of such functions has been well-studied (e.g. \cite{Darst,Darst2,Eidswick,Falconer2,JKPS,KS1,Troscheit}), we will focus exclusively on the case $a>1/(N+1)$, when $F_{N,a}$ is of unbounded variation. However, see \cite{Allaart} for a detailed analysis of the case $N=1$, $a\leq 1/2$.

Our main goal is to study the finite and infinite derivatives of $F_{N,a}$, thereby extending results of Okamoto \cite{Okamoto} and Allaart \cite{Allaart}. We first show that for each $N$ the differentiability of $F_{N,a}$ follows the trichotomy discovered by Okamoto \cite{Okamoto} for the case $N=1$: There are thresholds $\tilde{a}_0$ and $a_0^*$ (depending on $N$) such that $F_{N,a}$ is nowhere differentiable for $a\geq a_0^*$; nondifferentiable almost everywhere for $\tilde{a}_0\leq a<a_0^*$; and differentiable almost everywhere for $a<\tilde{a}_0$. We moreover compute the Hausdorff dimension of the exceptional sets implicit in the above statement.

In \cite{Allaart} a surprising connection was found between the infinite derivatives of $F_{1,a}$ and expansions of real numbers in noninteger bases. Our aim here is to generalize this result. We first give an explicit description of the set $\DD_\infty(a)$ of points $x$ for which $F_{N,a}'(x)=\pm\infty$, and then show that this set is closely related to the set $\mathcal{A}_\beta$ of real numbers which have a unique expansion in base $\beta$ over the alphabet $\{0,1,\dots,N\}$, where $\beta=1/a$. This allows us to express the Hausdorff dimension of $\DD_\infty(a)$ directly in terms of the dimension of $\mathcal{A}_\beta$, which is known to vary continuously with $\beta$ and can be calculated explicitly for many values of $\beta$; see \cite{KKL,KongLi}. We also take advantage of other recent breakthroughs in the theory of $\beta$-expansions to obtain a complete classification of the cardinality of $\DD_\infty(a)$.

To end this introduction, we point out that the box-counting dimension of the graph of $F_{N,a}$ is given by
\begin{equation*}
\dim_B \graph(F_{N,a})=1+\frac{\log\big(2(N+1)a-1\big)}{\log(2N+1)}, \qquad \frac{1}{N+1}<a<1.
\end{equation*}
This follows easily from the self-affine structure of $F_{N,a}$, for instance by using Example 11.4 of Falconer \cite{Falconer}. The Hausdorff dimension, on the other hand, does not appear to be known for any value of $a>1/(N+1)$, even when $N=1$; see the remark in the introduction of \cite{Allaart}.

\section{Main results} \label{sec:results}

\subsection{Finite derivatives} \label{subsec:finite}

Following standard convention, we consider a function $f$ to be {\em differentiable} at a point $x$ if $f$ has a well-defined {\em finite} derivative at $x$. For each $N\in\NN$, let $a_{\min}:=a_{\min}(N):=1/(N+1)$, let
\begin{equation}
a_0^*:=a_0^*(N):=\frac{3N+1}{(N+1)(2N+1)},
\end{equation}
and let $\tilde{a}_0:=\tilde{a}_0(N)$ be the unique root in $(a_{\min},1)$ of the polynomial equation
\begin{equation}
(2N+1)^{2N+1}a^{N+1}\big((N+1)a-1\big)^N=N^N.
\label{eq:defining-a0}
\end{equation}
The first ten values of $\tilde{a}_0(N)$ are shown in Table \ref{tab:thresholds} below. Asymptotically, $\tilde{a}_0(N)\sim (1+\sqrt{2})/2N$ as $N\to\infty$; see Proposition \ref{prop:thresholds} below.

The case $N=1$ of the following theorem is due to Okamoto \cite{Okamoto}, with the exception of the boundary value $a=\tilde{a}_0$, which was addressed by Kobayashi \cite{Kobayashi}.


\begin{theorem} \label{thm:differentiability}
\begin{enumerate}[(i)]
\item If $a_0^*\leq a<1$, then $F_{N,a}$ is nowhere differentiable;
\item If $\tilde{a}_0\leq a<a_0^*$, then $F_{N,a}$ is nondifferentiable almost everywhere, but is differentiable at uncountably many points;
\item If $a_{\min}<a<\tilde{a}_0$, then $F_{N,a}$ is differentiable almost everywhere, but is nondifferentiable at uncountably many points.
\end{enumerate}

\bigskip
Moreover, if $F_{N,a}$ is differentiable at a point $x\in(0,1)$, then $F_{N,a}'(x)=0$.
\end{theorem}


Statements (ii) and (iii) of the above theorem involve uncountably large sets of Lebesgue measure zero; it is of interest to determine their Hausdorff dimension. Let
\begin{equation*}
\DD_0(a):=\DD_0^{(N)}(a):=\{x\in(0,1): F_{N,a}'(x)=0\}.
\end{equation*}
Define the functions
\begin{equation*}
\phi_N(a):=\frac{\log\big((2N+1)a\big)}{\log(Na)-\log\big((N+1)a-1\big)}, \qquad a_{\min}<a<1,
\end{equation*}
and
\begin{equation*}
h_N(p):=-\frac{1}{\log(2N+1)}\left[p\log\left(\frac{p}{N}\right)+(1-p)\log\left(\frac{1-p}{N+1}\right)\right], \qquad 0<p<1.
\end{equation*}
For a set $E\subset \RR$, we denote the Hausdorff dimension of $E$ by $\dim_H E$.

\begin{theorem} \label{thm:Hdim}
\begin{enumerate}[(i)]
\item If $\tilde{a}_0\leq a<a_0^*$, then $\dim_H \DD_0^{(N)}(a)=h_N(\phi_N(a))$;
\item If $a_{\min}<a\leq\tilde{a}_0$, then $\dim_H \big((0,1)\backslash\DD_0^{(N)}(a)\big)=h_N(\phi_N(a))$;
\item The function $a\mapsto h_N(\phi_N(a))$ is concave on $a_{\min}<a<a_0^*$, takes on its maximum value of $1$ at $a=\tilde{a}_0$, and its limits as $a\downarrow a_{\min}$ and as $a\uparrow a_0^*$ are $\log(N+1)/\log(2N+1)$ and $\log N/\log(2N+1)$, respectively. 
\end{enumerate}
\end{theorem}

Observe that for $N\geq 2$, $\dim_H \DD_0(a)$ has a discontinuity at $a_0^*$, where it jumps from $\log N/\log(2N+1)>0$ to $0$; see Figure \ref{fig:Hausdorff-dimensions}.

\begin{figure}
\begin{center}
\epsfig{file=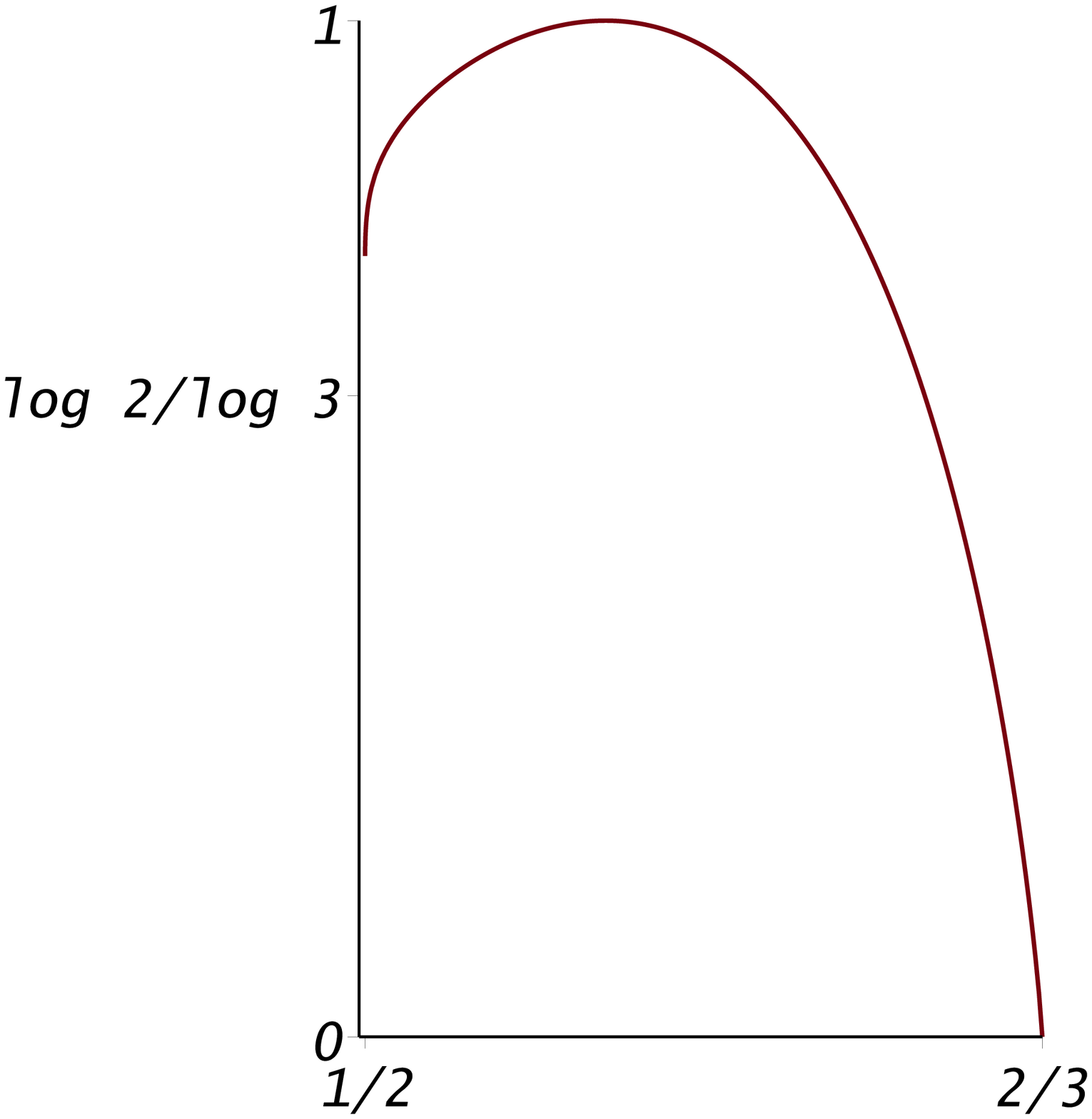, height=.25\textheight, width=.4\textwidth} \qquad\quad
\epsfig{file=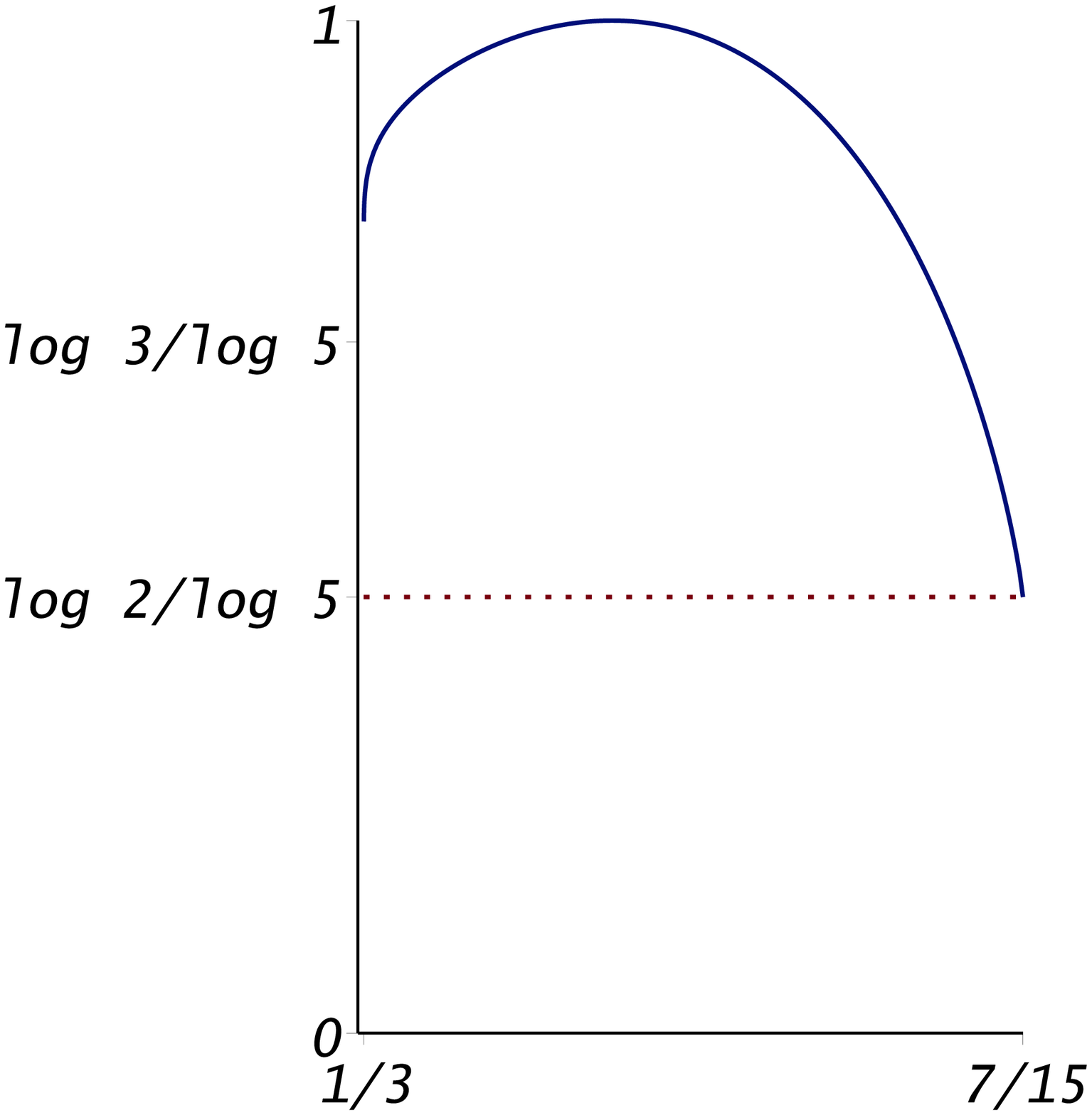, height=.25\textheight, width=.4\textwidth}
\caption{Graph of the function $h_N(\phi_N(a))$ from Theorem \ref{thm:Hdim} for $N=1$ (left) and $N=2$ (right). While not clearly visible, the limits at the left endpoints are $\log 2/\log 3$ and $\log 3/\log 5$, respectively.}
\label{fig:Hausdorff-dimensions}
\end{center}
\end{figure}

\begin{table}
\renewcommand{\arraystretch}{1.3}
\[
\begin{array}{r||c|c|c|c|c}
N & a_{\min}(N) & \tilde{a}_0(N) & a_0^*(N) & \hat{a}_\infty(N) & a_\infty^*(N)\\ \hline
1 & \frac12=.5000 & .5592 & \frac23\approx .6667 & .5598 & .6180 \\
2 & \frac13 \approx .3333 & .3835 & \frac{7}{15}\approx .4667 & .4047 & \frac12=.5000\\
3 & \frac14=.2500 & .2914 & \frac{5}{14}\approx .3571 & .3444 & .3660\\
4 & \frac15=.2000 & .2349 & \frac{13}{45}\approx .2889 & .2728 & \frac13\approx .3333\\
5 & \frac16\approx .1667 & .1967 & \frac{8}{33}\approx .2424 & .2534 & .2638\\
6 & \frac17\approx .1429 & .1692 & \frac{19}{91}\approx .2088 & .2104 & \frac14=.2500\\
7 & \frac18=.1250 & .1484 & \frac{11}{60}\approx .1833 & .2014 & .2071\\
8 & \frac19\approx .1111 & .1321 & \frac{25}{153}\approx .1634 & .1724 & \frac15=.2000\\
9 & \frac{1}{10}=.1000 & .1191 & \frac{14}{95}\approx .1474 & .1674 & .1708\\
10 & \frac{1}{11}\approx .0909 & .1084 & \frac{31}{231}\approx .1342 & .1463 & \frac16\approx .1667\\ \hline
\end{array}
\]
\caption{Five important thresholds for $a$ determining differentiability of $F_{N,a}$}
\label{tab:thresholds}
\end{table}

\subsection{Infinite derivatives} \label{subsec:infinite}

For $x\in[0,1)$, let 
\begin{equation*}
x=0.\xi_1\xi_2\xi_3\cdots=\sum_{i=1}^\infty \xi_i(2N+1)^{-i}
\end{equation*}
denote the expansion of $x$ in base $2N+1$, so $\xi_i\in\{0,1,\dots,2N\}$ for each $i$. When $x$ has two such expansions, we take the one ending in all zeros. Let 
$$M(x):=\#\{i\in\NN:\xi_i\ \mbox{is odd}\}.$$ 
We also write $\omega_i:=\xi_i/2$, and note that when $\xi_i$ is even, $\omega_i\in A:=\{0,1,\dots,N\}$. For $d\in A$, write $\bar{d}:=N-d$.

\begin{theorem} \label{thm:infinite-derivatives}
Let $a_{\min}<a<1$. A point $x\in(0,1)$ satisfies $F_{N,a}'(x)=\pm\infty$ if and only if $M(x)<\infty$ and the following two limits hold:
\begin{equation}
\lim_{n\to\infty}\big((2N+1)a\big)^n\left(1-\sum_{j=1}^\infty a^j\omega_{n+j}\right)=\infty,
\label{eq:for-right-derivative}
\end{equation}
and
\begin{equation}
\lim_{n\to\infty}\big((2N+1)a\big)^n\left(1-\sum_{j=1}^\infty a^j\overline{\omega_{n+j}}\right)=\infty.
\label{eq:for-left-derivative}
\end{equation}
Assuming all these conditions are satisfied, $F_{N,a}'(x)=\infty$ if $M(x)$ is even, and $F_{N,a}'(x)=-\infty$ if $M(x)$ is odd.
\end{theorem}

While the conditions \eqref{eq:for-right-derivative} and \eqref{eq:for-left-derivative} may look complicated at first sight, readers familiar with $\beta$-expansions will recognize the summations appearing in them. For a real number $1<\beta<N+1$ and $x\in\RR$, we call an expression of the form
\begin{equation}
x=\sum_{j=1}^\infty \frac{\omega_j}{\beta^j}, \qquad\mbox{where $\omega_1,\omega_2,\dots \in A$}
\label{eq:beta-expansion}
\end{equation}
an {\em expansion of $x$ in base $\beta$ over the alphabet $A$} (or simply, a {\em $\beta$-expansion}). Clearly such an expansion exists if and only if $0\leq x\leq N/(\beta-1)$. It is well known (see \cite{Sidorov}) that almost every $x$ in this interval has a continuum of $\beta$-expansions. For the purpose of this article, we reduce the interval a bit further and consider the so-called {\em univoque set}
\begin{equation*}
\A_\beta:=\{x\in ((N-\beta+1)/(\beta-1),1): x\ \mbox{has a {\em unique} expansion of the form \eqref{eq:beta-expansion}}\}.
\end{equation*}
Let $\Omega:=A^\NN$. For $\beta>1$, let $\Pi_\beta:\Omega\to \RR$ denote the projection map given by
\begin{equation*}
\Pi_\beta(\omega)=\sum_{j=1}^\infty \frac{\omega_j}{\beta^j}, \qquad \omega=\omega_1\omega_2\cdots \in\Omega,
\end{equation*}
so that \eqref{eq:beta-expansion} can be written compactly as $x=\Pi_\beta(\omega)$. Let $\sigma$ denote the left shift map on $\Omega$; that is, $\sigma(\omega_1\omega_2\cdots)=\omega_2\omega_3\cdots$. Define the set
\begin{equation*}
\UU_\beta:=\Pi_\beta^{-1}(\A_\beta).
\end{equation*}
It is essentially due to Parry \cite{Parry} (see also \cite[Lemma 5.1]{Allaart}) that
\begin{equation*}
\omega\in\UU_\beta \qquad \Longleftrightarrow \qquad \Pi_\beta(\sigma^n(\omega))<1\ \mbox{and}\ \Pi_\beta(\sigma^n(\bar{\omega}))<1\ \mbox{for all $n\geq 0$},
\end{equation*}
and this, together with Theorem \ref{thm:infinite-derivatives}, suggests a close connection between the set
\begin{equation*}
\DD_\infty(a):=\DD_\infty^{(N)}(a):=\{x\in(0,1): F_{N,a}'(x)=\pm\infty\}
\end{equation*}
and the univoque set $\A_\beta$, where $\beta=1/a$. The size of $\A_\beta$ has been well-studied, starting with the remarkable theorem of Glendinning and Sidorov \cite{GlenSid}. There are two pertinent thresholds, which we now define. First, for $N\in\NN$, let
\begin{equation*}
G(N):=\begin{cases}
m+1 & \mbox{if $N=2m$},\\
\frac{m+\sqrt{m^2+4m}}{2} & \mbox{if $N=2m-1$}.
\end{cases}
\end{equation*}
Baker \cite{Baker} calls $G(N)$ a {\em generalized golden ratio}, because $G(1)=(1+\sqrt{5})/2$.

Next, recall that the {\em Thue-Morse sequence} is the sequence $(\tau_j)_{j=0}^\infty$ of $0$'s and $1$'s given by $\tau_j=s_j \mod 2$, where $s_j$ is the number of $1$'s in the binary representation of $j$. Thus,
\begin{equation*}
(\tau_j)_{j=0}^\infty=0110\ 1001\ 1001\ 0110\ 1001\ 0110\ 0110\ 1001\ \dots
\end{equation*}
For each $N\in\NN$, define a generalized Thue-Morse sequence $\tau^{(N)}=\big(\tau_i^{(N)}\big)_{i=1}^\infty$ by
\begin{equation*}
\tau_i^{(N)}:=\begin{cases}
m-1+\tau_i & \mbox{if $N=2m-1$},\\
m+\tau_i-\tau_{i-1} & \mbox{if $N=2m$}.
\end{cases}
\end{equation*}
Finally, let $\beta_c:=\beta_c(N)$ be the {\em Komornik-Loreti constant} \cite{KomLor,KomLor2}; that is, $\beta_c$ is the unique positive value of $\beta$ such that
\begin{equation*}
\Pi_\beta\left(\tau^{(N)}\right)=1.
\end{equation*}
The following theorem is due to Glendinning and Sidorov \cite{GlenSid} for $N=1$, and to Kong et al. \cite{KLD} and Baker \cite{Baker} for $N\geq 2$. 

\begin{theorem} \label{thm:folklore}
The set $\A_\beta$ is:
\begin{enumerate}[(i)]
\item empty if $\beta\leq G(N)$;
\item nonempty but countable if $G(N)<\beta<\beta_c(N)$;
\item uncountable but of Hausdorff dimension zero if $\beta=\beta_c(N)$;
\item of positive Hausdorff dimension if $\beta>\beta_c(N)$.
\end{enumerate}
\end{theorem}

(In case (ii), there is a further threshold between $G(N)$ and $\beta_c(N)$ that separates finite and infinite cardinalities of $\A_\beta$, but this is not relevant to our present aims.)

Now let $a_\infty^*:=a_\infty^*(N):=1/G(N)$, and $\hat{a}_\infty:=\hat{a}_\infty(N):=1/\beta_c(N)$. For a finite set $S$, let $S^*$ denote the set of all finite sequences of elements of $S$, including the empty sequence.



\begin{theorem} \label{thm:close-correspondence}
\begin{enumerate}[(i)]
\item For all $a\in(\hat{a}_\infty,1)$ and for almost all $a\in(a_{\min},\hat{a}_\infty)$, 
\begin{equation}
\DD_\infty(a)=\left\{\Pi_{2N+1}(v\cdot 2\omega): v\in \{0,1,\dots,2N\}^*, \omega\in\UU_{1/a}\right\},
\label{eq:set-correspondence}
\end{equation}
where $2\omega:=(2\omega_1,2\omega_2,\dots)$ and $v\cdot 2\omega$ denotes the concatenation of $v$ with $2\omega$;
\item For all $a\in(a_{\min},1)$, $\DD_\infty(a)$ is a subset of the set in the right-hand-side of \eqref{eq:set-correspondence}, and the inclusion is proper for infinitely many $a\in(a_{\min},\hat{a}_\infty]$, including $\hat{a}_\infty$ itself;
\item For all $a\in(a_{\min},1)$,
\begin{equation}
\dim_H \DD_\infty(a)=\frac{\log(1/a)}{\log(2N+1)} \dim_H \A_{1/a}.
\label{eq:dimension-correspondence}
\end{equation}
\end{enumerate}
\end{theorem}

Theorem \ref{thm:close-correspondence}(i) says that for almost all $a$, the set $\DD_\infty(a)$ consists precisely of those points whose base $2N+1$ expansion is obtained by taking an arbitrary point $x$ having a unique expansion in base $\beta$ (where $\beta=1/a$), doubling the base $\beta$ digits of $x$, and appending the resulting sequence to an arbitrary finite prefix of digits from $\{0,1,\dots,2N\}$.

\begin{corollary} \label{thm:size-of-Dinf}
The set $\DD_\infty(a)$ is:
\begin{enumerate}[(i)]
\item empty if $a\geq a_\infty^*$;
\item countably infinite, containing only rational points, if $\hat{a}_\infty<a<a_\infty^*$;
\item uncountable with Hausdorff dimension zero if $a=\hat{a}_\infty$;
\item of strictly positive Hausdorff dimension if $a_{\min}<a<\hat{a}_\infty$.
\end{enumerate}
Moreover, on the interval $(a_{\min},\hat{a}_\infty)$, the function $a\mapsto \dim_H\DD_\infty(a)$ is continuous and nonincreasing, and its points of decrease form a set of Lebesgue measure zero. 
\end{corollary}

Regarding the relative ordering and the asymptotics of the five thresholds in Table \ref{tab:thresholds}, we have the following:

\begin{proposition} \label{prop:thresholds}
\begin{enumerate}[(i)]
\item For each $N\geq 5$, we have
\begin{equation}
a_{\min}(N)<\tilde{a}_0(N)<a_0^*(N)<\hat{a}_\infty(N)<a_\infty^*(N). 
\end{equation}
\item As $N\to\infty$, we have
\begin{gather*}
Na_{\min}(N)\to 1, \qquad N\tilde{a}_0(N)\to \frac{1+\sqrt{2}}{2}=1.207\cdots,\\
Na_0^*(N)\to \frac{3}{2}, \qquad N\hat{a}_\infty(N)\to 2, \qquad Na_\infty^*(N)\to 2.
\end{gather*}
\end{enumerate}
\end{proposition}

It is interesting to observe that for $N=1$, there is an interval of $a$-values (namely, $a_\infty^*<a<a_0^*$) for which $\dim_H \DD_0(a)>0$ but $\DD_\infty(a)=\emptyset$. In other words, for such $a$ there are uncountably many points where $F_{N,a}$ is differentiable, but no points where it has an infinite derivative. For $2\leq N\leq 4$ there is no such $a$, but there is still an interval (namely, $\hat{a}_\infty<a<a_0^*$) for which $\dim_H \DD_0(a)>0$ but $\DD_\infty(a)$ is only countable. For all $N\geq 5$, however, $\dim_H \DD_\infty(a)>0$ whenever $\dim_H \DD_0(a)>0$.

\section{Proofs of Theorems \ref{thm:differentiability} and \ref{thm:Hdim}} \label{sec:finite-derivatives}

Recall that for $x\in[0,1)$, $x=0.\xi_1\xi_2\xi_3\cdots$ denotes the expansion of $x$ in base $2N+1$. We first introduce some additional notation. For $n\in\NN$, let $i(n):=i(n;x)$ denote the number of odd digits among $\xi_1,\dots,\xi_n$. Let $x_{n,j}:=j/(2N+1)^n$, and put $I_{n,j}:=[x_{n,j},x_{n,j+1})$ for $n\in\NN\cup\{0\}$ and $j\in\ZZ$. For $x\in(0,1)$ and $n\in\NN\cup\{0\}$, let $I_n(x)$ denote that interval $I_{n,j}$ which contains $x$.

The first important observation is that
\begin{equation}
F_{N,a}(x_{n,j})=f_n(x_{n,j}), \qquad \mbox{for $j=0,1,\dots,(2N+1)^n$}.
\label{eq:fixed-forever}
\end{equation}

Next, recall that $b$ is the number such that $(N+1)a-Nb=1$. The recursive construction of the piecewise linear approximants $f_n$ implies that
\begin{equation}
f_n'(x)=(2N+1)^n a^{n-i(n)}(-b)^{i(n)},
\label{eq:slopes}
\end{equation}
at all $x$ not of the form $x_{n,j}$, $j\in\ZZ$. As a result,
\begin{equation*}
\frac{f_{n+1}'(x)}{f_n'(x)}\in \{(2N+1)a,-(2N+1)b\},
\end{equation*}
and since neither of these two values equals $1$, $f_n'(x)$ cannot converge to a nonzero finite number. Clearly, if $F_{N,a}'(x)$ exists, it must be equal to $\lim_{n\to\infty}f_n'(x)$ in view of \eqref{eq:fixed-forever}. The only possible finite value of $F_{N,a}'(x)$, therefore, is zero.

\begin{lemma} \label{lem:zero-derivative}
For $x\in(0,1)$, $F_{N,a}'(x)=0$ if and only if $\lim_{n\to\infty}f_n'(x)=0$.
\end{lemma}

\begin{proof}
Only the ``if" part requires proof. For simplicity, write $F:=F_{N,a}$. Let $s_{n,j}$ denote the slope of $f_n$ on the interval $I_{n,j}$. An easy induction argument shows that
\begin{equation}
\frac{s_{n,j+1}}{s_{n,j}}\in\left\{-\frac{a}{b},-\frac{b}{a}\right\}, \qquad j=0,1,\dots,(2N+1)^n-2.
\label{eq:consecutive-slope-ratios}
\end{equation}
Furthermore,
\begin{equation}
x\in I_{n,j} \quad \Rightarrow \quad \min\{F(x_{n,j}),F(x_{n,j+1})\}\leq F(x) \leq \max\{F(x_{n,j}),F(x_{n,j+1})\}.
\label{eq:stay-in-box}
\end{equation}
Now assume $f_n'(x)\to 0$. Fix $h>0$, let $n$ be the integer such that $(2N+1)^{-n-1}<h\leq (2N+1)^{-n}$, and let $j$ be such that $I_{n,j}=I_n(x)$. Then $x_{n,j}\leq x<x_{n,j+1}$ and $x_{n,j}\leq x+h<x_{n,j+2}$. If $x+h>x_{n,j+1}$, \eqref{eq:stay-in-box} gives
\begin{align*}
|F(x+h)-F(x)|&\leq |F(x+h)-F(x_{n,j+1})|+|F(x_{n,j+1})-F(x)|\\
&\leq |F(x_{n,j+2})-F(x_{n,j+1})|+|F(x_{n,j+1})-F(x_{n,j})|\\
&=(2N+1)^{-n}\left(|s_{n,j+1}|+|s_{n,j}|\right)\\
&\leq (1+C)(2N+1)^{-n}|s_{n,j}|\\
&=(1+C)(2N+1)^{-n}|f_n'(x)|,
\end{align*}
where $C:=\max\{a/b,b/a\}$, and the last inequality follows from \eqref{eq:consecutive-slope-ratios}. If $x+h\leq x_{n,j+1}$, the same bound follows even more directly. Thus, we obtain the estimate
\begin{equation*}
\left|\frac{F(x+h)-F(x)}{h}\right|\leq (2N+1)(1+C)|f_n'(x)|,
\end{equation*}
showing that $F$ has a vanishing right derivative at $x$. By a similar argument, $F$ has a vanishing left derivative at $x$ as well, and hence, $F'(x)=0$. 
\end{proof}

Now define
\begin{equation*}
l(x):=\liminf_{n\to\infty} \frac{i(n;x)}{n}, \qquad x\in (0,1),
\end{equation*}
and use \eqref{eq:slopes} to write
\begin{equation*}
|f_n'(x)|=\left[(2N+1)a\left(\frac{b}{a}\right)^{i(n)/n}\right]^n.
\end{equation*}
The significance of the function $\phi_N(a)$ is that
\begin{equation*}
(2N+1)a\left(\frac{b}{a}\right)^{\phi_N(a)}=1.
\end{equation*}
Since $0<b/a<1$, this last equation together with Lemma \ref{lem:zero-derivative} implies
\begin{equation}
\left\{x\in(0,1): l(x)>\phi_N(a)\right\} \subseteq \DD_0(a) \subseteq 
\left\{x\in(0,1): l(x)\geq\phi_N(a)\right\}.
\label{eq:frequency-sandwich}
\end{equation}

\begin{proof}[Proof of Theorem \ref{thm:differentiability}]
(i) Assume first that $a\geq a_0^*$. Since $a>b>0$, \eqref{eq:slopes} yields
\begin{equation*}
|f_n'(x)|\geq \big((2N+1)b\big)^n=\left[\frac{2N+1}{N}\big((N+1)a-1\big)\right]^n\geq 1
\end{equation*}
for every $x$ not of the form $k/(2N+1)^n$, so $f_n'(x)\not\to 0$ for such $x$. Hence, $F_{N,a}$ is nowhere differentiable.

(ii) and (iii): By Borel's normal number theorem,
\begin{equation}
\lim_{n\to\infty} \frac{i(n;x)}{n}=\frac{N}{2N+1} \qquad\mbox{for almost every $x\in(0,1)$}.
\label{eq:normal-frequency-of-odd}
\end{equation}
By definition of $\tilde{a}_0$, we have $\phi_N(\tilde{a}_0)=N/(2N+1)$. Moreover, $\phi_N$ is monotone increasing on $(a_{\min},a_0^*)$. From these observations, it follows via \eqref{eq:frequency-sandwich} and \eqref{eq:normal-frequency-of-odd} that $\DD_0(a)$ has Lebesgue measure one if $a_{\min}<a<\tilde{a}_0$, and Lebesgue measure zero if $a>\tilde{a}_0$.
Finally, the law of the iterated logarithm implies that for almost every $x\in(0,1)$, $i(n;x)/n<N/(2N+1)=\phi_N(\tilde{a}_0)$, and therefore $|f_n'(x)|\geq 1$, for infinitely many $n$. (See \cite{Kobayashi} for more details in the case $N=1$). Thus, $\DD_0(\tilde{a}_0)$ has measure zero as well. The remaining statements follow from Theorem \ref{thm:Hdim}, which is proved below.
\end{proof}

In view of the relations \eqref{eq:frequency-sandwich}, we define for $p\in(0,1)$ the sets
\begin{gather*}
R_p:=\{x\in(0,1): l(x)>p\}, \qquad \bar{R}_p:=\{x\in(0,1): l(x)\geq p\},\\
S_p:=\{x\in(0,1): l(x)<p\}, \qquad \bar{S}_p:=\{x\in(0,1): l(x)\leq p\}.
\end{gather*}

\begin{lemma} \label{lem:dimensions}
We have
\begin{equation}
\dim_H R_p=\dim_H \bar{R}_p=\begin{cases}
1 & \mbox{if $0<p\leq N/(2N+1)$},\\
h_N(p) & \mbox{if $N/(2N+1)\leq p<1$},
\end{cases}
\label{eq:large-p-dimension}
\end{equation}
and
\begin{equation}
\dim_H S_p=\dim_H \bar{S}_p=\begin{cases}
h_N(p) & \mbox{if $0<p\leq N/(2N+1)$},\\
1 & \mbox{if $N/(2N+1)\leq p<1$}.
\end{cases}
\label{eq:small-p-dimension}
\end{equation}
\end{lemma}

\begin{proof}
We prove \eqref{eq:large-p-dimension}; the proof of \eqref{eq:small-p-dimension} is analogous. Since $\bar{R}_p\supseteq R_p\supseteq \bar{R}_{p-\eps}$ for all $p>\eps>0$ and $h_N$ is continuous in $p$, it suffices to compute $\dim_H \bar{R}_p$.
First define, for nonnegative real numbers $p_0,\dots,p_{2N}$ with $p_0+\dots+p_{2N}=1$, the set
\begin{equation*}
\FF(p_0,\dots,p_{2N}):=\left\{x\in(0,1): \lim_{n\to\infty}\frac{\#\{j\leq n:\xi_j(x)=i\}}{n}=p_i, \quad i=0,1,\dots,2N\right\}.
\end{equation*}
It is well known (see, for instance, \cite[Proposition 10.1]{Falconer}) that
\begin{equation}
\dim_H \FF(p_0,\dots,p_{2N})=-\frac{1}{\log(2N+1)}{\sum_{i=0}^{2N}p_i\log p_i}.
\label{eq:entropy-dimension}
\end{equation}
If $0<p\leq N/(2N+1)$, then $\bar{R}_p$ has Lebesgue measure one by \eqref{eq:normal-frequency-of-odd}. Suppose $N/(2N+1)<p<1$. Then $\bar{R}_p$ contains the set
\begin{equation*}
\FF\left(\frac{1-p}{N+1},\frac{p}{N},\frac{1-p}{N+1},\dots,\frac{p}{N},\frac{1-p}{N+1}\right),
\end{equation*}
which has Hausdorff dimension $h_N(p)$ by \eqref{eq:entropy-dimension}. Therefore, $\dim_H \bar{R}_p\geq h_N(p)$.

For the reverse inequality, we introduce a probability measure $\mu$ on $(0,1)$ as follows. Set
\begin{equation*}
\mu(I_n(x)):=\left(\frac{p}{N}\right)^{i(n;x)}\left(\frac{1-p}{N+1}\right)^{n-i(n;x)}.
\end{equation*}
This defines $\mu(I_{n,j})$ for all $n\in\NN\cup\{0\}$ and $j=0,1,\dots,(2N+1)^n-1$ in such a way that $\mu(I_{n,j})=\sum_{\nu=0}^{2N}\mu(I_{n+1,(2N+1)j+\nu})$ for all $n$ and $j$, and hence $\mu$ extends uniquely to a Borel probability measure on $(0,1)$, which we again denote by $\mu$. It is a routine exercise that $\mu$ concentrates its mass on the set $\{x:\lim_{n\to\infty}i(n;x)/n=p\}$, so in particular $\mu(\bar{R}_p)=1$. It now follows just as in the proof of \cite[Lemma 4.2]{Allaart} that if $s>h_N(p)$, then
\begin{equation*}
\limsup_{n\to\infty} \frac{\mu(I_n(x))}{|I_n(x)|^s}=\infty,
\end{equation*}
where $|I_n(x)|$ denotes the length of $I_n(x)$. Using \cite[Proposition 4.9]{Falconer}, we conclude that $\dim_H \bar{R}_p\leq h_N(p)$.
\end{proof}

\begin{proof}[Proof of Theorem \ref{thm:Hdim}]
Statements (i) and (ii) follow immediately from Lemma \ref{lem:dimensions} and \eqref{eq:frequency-sandwich}. The proof of (iii) is a straightforward calculus exercise.
\end{proof}

\section{Proof of Theorem \ref{thm:infinite-derivatives}} \label{sec:infinite-derivative-proof}

To avoid notational clutter we again write $F:=F_{N,a}$. In order for $F$ to have an infinite derivative at $x$, it is clear that $f_n'(x)$ must tend to $\pm\infty$. By \eqref{eq:slopes}, this is the case if and only if $\xi_n$ is even for all but finitely many $n$. However, it turns out that this condition is not sufficient.

We begin with an infinite-series representation of $F$ (see \cite{Kobayashi} for a proof when $N=1$):
\begin{equation*}
F(x)=\sum_{n=1}^\infty a^{n-1-i(n-1)}(-b)^{i(n-1)}y_{\xi_n},
\end{equation*}
where $y_0,y_1,\dots,y_{2N}$ are the numbers used in the introduction to define $f_n$. In the special case when $\xi_n$ is even for every $n$, this reduces to 
\begin{equation}
F(x)=\frac{1}{N}\sum_{n=1}^\infty a^{n-1}(1-a)\omega_n,
\label{eq:simpler-series-expression}
\end{equation}
where $\omega_n:=\xi_n/2$, and we have used that $i(a-b)=i(1-a)/N$ for $i=0,1,\dots,N$.
Let $F_+'$ and $F_-'$ denote the right-hand and left-hand derivative of $F$, respectively.

\begin{lemma} \label{lem:basic-right-derivative}
Assume $\xi_n(x)$ is even for every $x$, and let $\omega_n:=\xi_n(x)/2$. Then $F_+'(x)=\infty$ if and only if 
\begin{equation}
\lim_{n\to\infty}\big((2N+1)a\big)^n\left(1-\sum_{k=1}^\infty a^k\omega_{n+k}\right)=\infty.
\label{eq:right-derivative-condition}
\end{equation}
\end{lemma}

\begin{proof}
For each $n\in\NN$, let $j_n$ be the integer such that $x\in I_{n,j_n}$, and put $z_n:=x_{n,j_n+2}$ (the right endpoint of $I_{n,j_n+1}$). In order that $F_+'(x)=\infty$, it is clearly necessary that
\begin{equation}
\lim_{n\to\infty}\frac{F(z_n)-F(x)}{z_n-x}=\infty.
\label{eq:special-difference-quotient}
\end{equation}
The slope of $f_n$ on $I_{n,j_n}$ is $\big((2N+1)a\big)^n$, and by \eqref{eq:consecutive-slope-ratios}, the slope of $f_n$ on $I_{n,j_n+1}$ is $-(2N+1)^n a^{n-1}b$, independent of $\xi_n$. Therefore, the difference $F(z_n)-F(x)$ does not depend on $\xi_n$, and we may assume $\xi_n=0$. Then $z_n=0.\xi_1\xi_2\cdots \xi_{n-1}200\cdots$, and \eqref{eq:simpler-series-expression} applied to $x$ and $z_n$ (noting that $\omega_n(x)=0$ and $\omega_n(z_n)=1$) gives
\begin{align*}
F(z_n)-F(x)&=a^{n-1}\frac{1-a}{N}-\sum_{k=n+1}^\infty a^{k-1}\frac{(1-a)\omega_k}{N}\\
&=\frac{a^{n-1}(1-a)}{N}\left(1-\sum_{k=1}^\infty a^k\omega_{n+k}\right).
\end{align*}
Since $1/(2N+1)^n<z_n-x\leq 2/(2N+1)^n$, it follows that \eqref{eq:special-difference-quotient} is equivalent to \eqref{eq:right-derivative-condition}, showing that \eqref{eq:right-derivative-condition} is necessary. We now demonstrate that it is also sufficient.

Assume \eqref{eq:right-derivative-condition}, or equivalently, \eqref{eq:special-difference-quotient}. Then $F(z_n)>F(x)$ for all large enough $n$. Given $h>0$, let $n\in\NN$ such that $(2N+1)^{-n-1}<h\leq (2N+1)^{-n}$, let $j:=j_n$, and as before, $z_n:=x_{n,j+2}$. If $x+h\geq x_{n,j+1}$, then, since $f_n'<0$ on $(x_{n,j+1},x_{n,j+2})$, we have by \eqref{eq:stay-in-box},
\begin{equation*}
\frac{F(x+h)-F(x)}{h}\geq \frac{F(z_n)-F(x)}{h}\geq \frac{F(z_n)-F(x)}{z_n-x} \to\infty,
\end{equation*}
where the last inequality holds for all sufficiently large $n$.

Assume now that $x+h<x_{n,j+1}$. Then $\xi_{n+1}=2i$ for some $i<N$, so $x\in I_{n+1,(2N+1)j+2i}$ and $z_{n+1}=x_{n+1,(2N+1)j+2i+2}<x_{n,j+1}$. Moreover, $x_{n+1,(2N+1)j+2i+1}<x+h<z_{n,j+1}$. In view of \eqref{eq:stay-in-box} and the zig-zag pattern in the graph of $f_{n+1}$, it follows that $F(x+h)\geq F(z_{n+1})$. Finally, $h\leq (2N+1)(z_{n+1}-x)$. Combining these facts, we obtain, for all sufficiently large $n$,
\begin{equation*} 
\frac{F(x+h)-F(x)}{h}\geq \frac{F(z_{n+1})-F(x)}{h}\geq \frac{F(z_{n+1})-F(x)}{(2N+1)(z_{n+1}-x)}.
\end{equation*}
Thus, by \eqref{eq:special-difference-quotient}, $F_+'(x)=\infty$.
\end{proof}

\begin{proof}[Proof of Theorem \ref{thm:infinite-derivatives}]
If $x=j/(2N+1)^n$ for some $n\in\NN$ and $j\in\ZZ$, then it follows immediately from \eqref{eq:consecutive-slope-ratios} that $F_+'(x)$ and $F_-'(x)$ are of opposite signs (in fact, one $+\infty$, the other $-\infty$), so $F$ does not have an infinite derivative at $x$. And since $\omega_n=\xi_n=0$ for all but finitely many $n$ in this case, 
\begin{equation*}
\sum_{j=1}^\infty a^j\overline{\omega_{n+j}}=\sum_{j=1}^\infty a^j N=\frac{aN}{1-a}>1
\end{equation*}
for all sufficiently large $n$, so \eqref{eq:for-left-derivative} fails.

Now assume that $x$ is {\em not} of the form $j/(2N+1)^n$, and let $m:=M(x)$. As already observed earlier, $F$ does not have an infinite derivative at $x$ if $m=\infty$, so assume $m<\infty$. Since $F(1-x)=1-F(x)$, it follows that $F_-'(x)=F_+'(1-x)$ when at least one of these two quantities exists, and moreover, $\xi_n(1-x)=2N-\xi_n(x)$ and so $\omega_n(1-x)=\overline{\omega_n(x)}$ when $\xi_n(x)$ is even. Therefore, it suffices to show that $F_+'(x)=\pm\infty$ if and only if \eqref{eq:for-right-derivative} holds. If $m=0$, this is immediate from Lemma \ref{lem:basic-right-derivative}, so assume $m>0$. Choose $n_0$ so that $\xi_n$ is even for all $n\geq n_0$, let $j$ be the integer such that $x\in I_{n_0,j}=[x_{n_0,j},x_{n_0,j+1})$. Then $x=x_{n_0,j}+(2N+1)^{-n_0}\hat{x}$, where $\hat{x}\in[0,1)$ satisfies the hypothesis of Lemma \ref{lem:basic-right-derivative}, and $M(x_{n_0,j})=M(x)=m$. Note that \eqref{eq:right-derivative-condition} holds for $\hat{x}$ if and only if it holds for $x$, since the condition is invariant under a shift of the sequence $(\xi_n)$. The graph of $F$ above $I_{n_0,j}$ is an affine copy of the full graph of $F$, scaled horizontally by $(2N+1)^{-n_0}$ and vertically by $a^{n_0-m}b^m$, and reflected top-to-bottom if $m$ is odd. Thus, $F_+'(x)$ is infinite if and only if $F_+'(\hat{x})$ is, with the same sign when $m$ is even, and the opposite sign when $m$ is odd. 
\end{proof}

\section{Proofs of Theorem \ref{thm:close-correspondence} and Corollary \ref{thm:size-of-Dinf}} \label{sec:using-beta-expansions}

\begin{proof}[Proof of Theorem \ref{thm:close-correspondence}]
We will need the auxiliary sets
\begin{equation}
\widetilde{\UU}_\beta:=\{\omega\in\UU_\beta: \limsup_{n\to\infty}\Pi_\beta(\sigma^n(\omega))<1\ \mbox{and}\ \limsup_{n\to\infty}\Pi_\beta(\sigma^n(\bar{\omega}))<1\},
\label{eq:alternative-representation-of-U-tilde}
\end{equation}
for $1<\beta<N+1$, as well as the family of affine maps
\begin{equation*}
\psi_{n,k}(x):=(2N+1)^{-n}(x+k), \qquad n\in\NN, \quad k=0,1,\dots,(2N+1)^n-1,
\end{equation*}
and the function $\Phi:\Omega\to[0,1]$ given by
\begin{equation}
\Phi(\omega):=2\Pi_{2N+1}(\omega), \qquad\omega\in\Omega.
\label{eq:Phi-definition}
\end{equation}
Since $a>a_{\min}$ implies that $\big((2N+1)a\big)^n\to\infty$, it follows from Theorem \ref{thm:infinite-derivatives} that
\begin{equation}
\bigcup_{n,k}\psi_{n,k}\big(\Phi\big(\widetilde{\UU}_{1/a}\big)\big)\subset \DD_\infty(a)\subset \bigcup_{n,k}\psi_{n,k}(\Phi(\UU_{1/a})),
\label{eq:key-sandwich}
\end{equation}
where the unions are over $n\in\NN$ and $k=0,1,\dots,(2N+1)^n-1$. Note that the set on the far right of \eqref{eq:key-sandwich} is precisely the set on the right hand side of \eqref{eq:set-correspondence}. It was shown in \cite{Allaart2} that $\widetilde{\UU}_\beta=\UU_\beta$ for all $1<\beta<\beta_c(N)$ and almost all $\beta_c(N)<\beta<N+1$, so \eqref{eq:key-sandwich} yields (i) and the first part of (ii). It was further shown in \cite{Allaart2} that there are infinitely many values of $\beta$, including $\beta_c(N)$, for which $\widetilde{\UU}_\beta$ is a proper subset of $\UU_\beta$, and that for each such $\beta$ and any given sequence $(\theta_n)$ of positive numbers, there are in fact uncountably many $\omega\in\UU_\beta\backslash \widetilde{\UU}_\beta$ such that
\begin{equation*}
\liminf_{n\to\infty} \theta_n \big(1-\Pi_\beta(\sigma^n(\omega))\big)<\infty.
\end{equation*}
Taking $\theta_n=\big((2N+1)a\big)^n$ we obtain the second part of statement (ii). 

To prove (iii), we consider Hausdorff dimension in the sequence space $\Omega$. For each $\beta>1$, define a metric $\rho_\beta$ on $\Omega$ by $\rho_\beta(\omega,\eta):=\beta^{-\inf\{i:\omega_i\neq\eta_i\}}$ for $\omega=\omega_1\omega_2\cdots$ and $\eta=\eta_1\eta_2\cdots$. To avoid confusion, we reserve the notation $\dim_H$ for Hausdorff dimension in $\RR$ and write $\Dim^{(\beta)}_H$ for Hausdorff dimension in $\Omega$ induced by the metric $\rho_\beta$. Since for any two numbers $\beta_1,\beta_2\in(1,\infty)$ we have
\begin{equation*}
\rho_{\beta_2}(\omega,\eta)=\left(\rho_{\beta_1}(\omega,\eta)\right)^{\log \beta_2/\log \beta_1},
\end{equation*}
it follows in a straightforward manner that
\begin{equation}
\Dim_H^{(\beta_1)}E=\frac{\log\beta_2}{\log\beta_1}\Dim_H^{(\beta_2)}E, \qquad E\subset\Omega.
\label{eq:Hausdorff-dimension-equivalence}
\end{equation}
We now make two important observations:
\begin{enumerate}[(1)]
\item The restriction of $\Pi_\beta$ to $\UU_\beta$ is bi-Lipschitz with respect to the metric $\rho_\beta$ (see \cite[Lemma 2.7]{JSS} or \cite[Lemma 2.2]{Allaart2});
\item The map $\Pi_{2N+1}$ is bi-Lipschitz on all of $\Omega$ with respect to $\rho_{2N+1}$. (This follows because $\Pi_{2N+1}$ maps the Cantor space $\Omega$ onto a geometric Cantor set in $[0,1]$.)
\end{enumerate}
Since bi-Lipschitz maps preserve Hausdorff dimension, these observations and \eqref{eq:Hausdorff-dimension-equivalence} imply that for any $\beta\in(1,N+1)$,
\begin{align*}
\dim_H \Pi_{2N+1}(\UU_\beta)&=\Dim_H^{(2N+1)}\UU_\beta\\
&=\frac{\log\beta}{\log(2N+1)}\Dim_H^{(\beta)}\UU_\beta\\
&=\frac{\log\beta}{\log(2N+1)}\dim_H \A_\beta.
\end{align*}
Now it was shown in \cite{Allaart2} that $\Dim_H^{(\beta)}\UU_\beta=\Dim_H^{(\beta)}\widetilde{\UU_\beta}$ for all $1<\beta<N+1$. Thus, taking $\beta=1/a$ and using \eqref{eq:Phi-definition}, \eqref{eq:key-sandwich} and the countable stability of Hausdorff dimension, we obtain \eqref{eq:dimension-correspondence}.
\end{proof}

\begin{proof}[Proof of Corollary \ref{thm:size-of-Dinf}]
From Theorems \ref{thm:folklore} and \ref{thm:close-correspondence} it follows immediately that $\DD_\infty(a)$ is empty when $a\geq a_\infty^*$, nonempty but countable when $\hat{a}_\infty<a<a_\infty^*$, of Hausdorff dimension zero when $a=\hat{a}_\infty$, and of positive Hausdorff dimension when $a_{\min}<a<\hat{a}_\infty$. The ``moreover" statement of the theorem is a consequence of Theorem \ref{thm:close-correspondence}(iii) and \cite[Theorem 2.6]{KongLi}. It remains to show that $\DD_\infty(a)$ contains only rational points when $\hat{a}_\infty<a<a_\infty^*$, and is uncountable when $a=\hat{a}_\infty$. The former follows from Theorem \ref{thm:close-correspondence}(ii) since, as pointed out in \cite{GlenSid,KLD}, $\UU_\beta$ contains only eventually periodic sequences when $\beta<\beta_c$; the latter is a direct consequence of Theorem \ref{thm:infinite-derivatives} and \cite[Theorem 1.3]{Allaart2}.
\end{proof}

\begin{proof}[Proof of Proposition \ref{prop:thresholds}]
(i) For the first two inequalities, let
\begin{equation*}
g_N(x):=N^{-N}(2N+1)^{2N+1}x^{N+1}\big((N+1)x-1\big)^N,
\end{equation*}
and observe that $g_N$ is strictly increasing for $x\geq 1/(N+1)$, with $g_N(\tilde{a}_0(N))=1$. Since $g_N(a_{\min}(N))=0$, this gives the first inequality. Now for a constant $c>1$, we can write
\begin{equation}
g_N(c/N)=2c\big(4c(c-1)\big)^N\left(1+\frac{1}{2N}\right)^{2N+1}\left(1+\frac{c}{(c-1)N}\right)^N.
\label{eq:g_N}
\end{equation}
Let $c_0:=(1+\sqrt{2})/2$. Then $4c_0(c_0-1)=1$, so \eqref{eq:g_N} shows that $g_N(c_0/N)\geq 2c_0>1$, and hence, $\tilde{a}_0(N)<c_0/N$. Straightforward algebra shows that $c_0/N<a_0^*(N)$ for every $N\geq 5$, establishing the second inequality. (Of course, by direct calculation, $\tilde{a}_0(N)<a_0^*(N)$ also for $N<5$.)

For the third inequality, observe that $\tau_i^{(2m)}\leq m+1$, so the definition of $\beta_c(N)$ implies, by summing a geometric series, that $\beta_c(2m)\leq m+2$. Routine algebra gives
\begin{equation*}
a_0^*(2m)=\frac{6m+1}{(2m+1)(4m+1)}<\frac{1}{m+2}\leq\hat{a}_\infty(2m),
\end{equation*}
for all $m\geq 4$. In addition, a direct calculation shows that $a_0^*(6)<\hat{a}_\infty(6)$; see Table \ref{tab:thresholds}. Similarly, $\tau_i^{(2m-1)}\leq m$, so that $\beta_c(2m-1)\leq m+1$, and again by routine algebra,
\begin{equation*}
a_0^*(2m-1)=\frac{3m-1}{m(4m-1)}<\frac{1}{m+1}\leq \hat{a}_\infty(2m-1),
\end{equation*}
for all $m\geq 3$. Thus, the third inequality holds for all $N\geq 5$. Finally, the last inequality follows since $G(N)<\beta_c(N)$; see Baker \cite{Baker}.

(ii) The limits involving $a_{\min}$, $a_0^*$ and $a_\infty^*$ are obvious, and the limit involving $\hat{a}_\infty$ was established by Baker \cite{Baker}, who gives a finer analysis of the asymptotics of the threshold $\beta_c(N)$. The second limit follows since, by \eqref{eq:g_N} and the aforementioned fact that $4c_0(c_0-1)=1$, $g_N(c/N)\to 0$ for $1<c<c_0$, and $g_N(c/N)\to \infty$ for $c>c_0$.
\end{proof}

\end{document}